\documentclass[a4paper,USenglish,cleveref]{lipics-v2021}

\usepackage{mathtools}
\usepackage{csquotes}
\usepackage{complexity}
\newclass{\EXPTIME}{EXPTIME}
\usepackage{bm}
\usepackage{xspace}

\nolinenumbers

\DeclareMathOperator{\girth}{girth}
\DeclareMathOperator{\dist}{dist}
\DeclarePairedDelimiter\abs{\lvert}{\rvert}
\newcommand{\N}{\mathbb{N}}

\newcommand{\Pralat}{Pra{\l}at\xspace}

\title{Cops and Robber -- When Capturing is not Surrounding}

\author{Paul Jungeblut}{Karlsruhe Institute of Technology, Germany}{paul.jungeblut@kit.edu}{https://orcid.org/0000-0001-8241-2102}{}

\author{Samuel Schneider}{Karlsruhe Institute of Technology, Germany}{samuel.schneider@kit.edu}{https://orcid.org/0009-0002-9680-4048}{}

\author{Torsten Ueckerdt}{Karlsruhe Institute of Technology, Germany}{torsten.ueckerdt@kit.edu}{https://orcid.org/0000-0002-0645-9715}{}

\authorrunning{P. Jungeblut, S. Schneider, T. Ueckerdt}

\Copyright{Paul Jungeblut, Samuel Schneider, Torsten Ueckerdt}

\ccsdesc[100]{Mathematics of computing $\rightarrow$ Discrete mathematics $\rightarrow$ Graph theory}

\keywords{Cops and Robber, Pursuit–evasion games, surrounding}

\relatedversion{}
\relatedversiondetails[cite=Jungeblut2023-Surround-WG]{An extended abstract of this work was presented at the 49th International Workshop on Graph-Theoretic Concepts in Computer Science (WG 2023)}{https://link.springer.com/chapter/10.1007/978-3-031-43380-1_29}

\hideLIPIcs

\begin{document}

\maketitle

\begin{abstract}
    We consider \enquote{surrounding} versions of the classic Cops and Robber game.
    The game is played on a connected graph in which two players, one controlling several cops and the other controlling a single robber, take alternating turns.
    In a turn, each player may move each of their pieces.
    The robber always moves between adjacent vertices.
    Regarding the moves of the cops, we distinguish four versions that differ in whether the cops are on the vertices or the edges of the graph and whether the robber may move on/through them.
    The goal of the cops is to surround the robber, i.e., to occupy all neighbors (vertex version) or incident edges (edge version) of the robber's current vertex.
    In contrast, the robber tries to avoid being surrounded indefinitely.
    Given a graph, the so-called \emph{cop number} denotes the minimum number of cops required to eventually surround the robber.

    We relate the different cop numbers of these versions, and prove that none of them is bounded by a function of the classical cop number and the maximum degree of the graph, thereby refuting a conjecture by Crytser, Komarov and Mackey~[Graphs and Combinatorics, 2020].
\end{abstract}

\section{Introduction}

\emph{Cops and Robber} is a well-known combinatorial game played by two players on a graph $G = (V,E)$.
The robber player controls a single robber, which we shall denote by~$r$, whereas the cop player controls~$k$ cops, denoted $c_1, \ldots, c_k$, for some specified integer~$k \geq 1$.
The players take alternating turns, and may perform one move with each of their pieces (the single robber or the~$k$ cops) in each turn.
In the classical game (and also many of its variants), the vertices of~$G$ are the possible positions for the pieces, while the edges of~$G$ model the possible moves.
Let us remark that no piece is forced to move, i.e., there is no \emph{zugzwang}.
On each vertex there can be any number of pieces.

The game begins with the cop player choosing vertices as the starting positions for the~$k$ cops $c_1, \ldots, c_k$.
Then, seeing the cops' positions, the robber player places~$r$ on a vertex of~$G$ as well.
The cop player wins if the cops \emph{capture} the robber, which in the classical version means that at least one cop stands on the same vertex as the robber.
On the other hand, the robber player wins if the robber can avoid being captured indefinitely.

The \emph{cop number}~$c(G)$ of a given connected\footnote{
    Cops cannot move between different connected components, so the cop number of any graph is the sum over all components.
    We thus consider connected graphs only.
} graph $G = (V,E)$ is the smallest~$k$ for which~$k$ cops can capture the robber in a finite number of turns.
Clearly, every graph satisfies $1 \leq c(G) \leq \abs{V}$.

We consider several versions of the Cops and Robber game.
In some of these, the cops are placed on the edges of~$G$ and allowed to move to an \emph{adjacent} edge (i.e., an edge sharing an endpoint) during their turn.
In all our versions, the robber acts as in the original game but loses the game if he is \emph{surrounded}\footnote{
    To distinguish between the classical and our versions, we use the term \emph{capture} to express that a cop occupies the same vertex as the robber.
    In contrast, a \emph{surround} always means that all neighbors, respectively incident edges, are occupied.
} by the cops, meaning that they have to occupy all adjacent vertices or incident edges.
At all times, let us denote by~$v_r$ the vertex currently occupied by the robber.
Specifically, we define the following versions of the game, each specifying the possible positions for the cops and the exact surrounding condition:

\begin{description}
    \item[Vertex Version]
    Cops are positioned on vertices of~$G$ (like the robber).
    They surround the robber if there is a cop on each neighbor of~$v_r$.
    Let~$c_V(G)$ denote the smallest number of cops needed to eventually surround the robber.

    \item[Edge Version]
    Cops are positioned on edges of~$G$.
    A cop on an edge~$e$ can move to any edge~$e'$ sharing an endpoint with~$e$ during its turn.
    The cops surround the robber if there is a cop on each edge incident to~$v_r$.
    Let~$c_E(G)$ denote the smallest number of cops needed to eventually surround the robber.
\end{description}

In both versions above, the robber sits on the vertices of~$G$ and moves along the edges of~$G$.
Due to the winning condition for the cops being a full surround, the robber may come very close to, say, a single cop without being threatened.
As this can feel counterintuitive, let us additionally consider a restrictive version of each game.
Here, we constrain the possible moves for the robber when cops are close by.
These \emph{restrictive} versions are given by the following rules:

\begin{description}
    \item[Restrictive Vertex Version]
    After the robber's turn, there may not be any cop on~$v_r$.
    In particular, the robber may not move onto a vertex occupied by a cop.
    Additionally, if a cop moves onto~$v_r$, then the robber must leave that vertex in his next turn.

    \item[Restrictive Edge Version.]
    The robber may not move along an edge that is currently occupied by a cop.
\end{description}

We denote the cop numbers of the restrictive versions by putting an additional \enquote{$\mathrm{r}$} in the subscript, i.e., the smallest number of cops needed to eventually surround the robber in these versions is $c_{V,\mathrm{r}}(G)$ and $c_{E,\mathrm{r}}(G)$, respectively.

Clearly, the restrictive versions are favorable for the cops, as they only restrict the robber.
Consequently, the corresponding cop numbers are always at most their nonrestrictive counterparts.
Thus, for every connected graph~$G$ we have
\begin{equation}
    \label{eqn:restrictive_smaller}
    c_{V, \mathrm{r}}(G) \leq c_{V}(G)
    \quad \text{and} \quad
    c_{E, \mathrm{r}}(G) \leq c_{E}(G)
    \text{.}
\end{equation}

\medskip
\noindent
A recent conjecture by Crytser, Komarov and Mackey~\cite{Crytser2020_Containment} states that the cop number in the restrictive edge version can be bounded from above by the classical cop number and the maximum degree of the graph:

\begin{conjecture}[\cite{Crytser2020_Containment}]
    \label{conj:edge_restrictive_vs_classic}
    For every connected graph~$G$ we have~$c_{E, \mathrm{r}}(G) \leq c(G) \cdot \Delta(G)$.
\end{conjecture}

\Pralat~\cite{Pralat2015_Containment} verified \cref{conj:edge_restrictive_vs_classic} for the random graph~$G(n,p)$, i.e., the graph on~$n$ vertices where each possible edge is chosen independently with probability~$p$, for some ranges of~$p$.
Let us note that \cref{conj:edge_restrictive_vs_classic}, if true, would strengthen a theorem by Crytser, Komarov and Mackey~\cite{Crytser2020_Containment} stating that~$c_{E, \mathrm{r}}(G) \leq \gamma(G) \cdot \Delta(G)$, where~$\gamma(G)$ denotes the size of a smallest dominating set in~$G$.

\subsection{Our Results}

Our main contribution is to disprove \cref{conj:edge_restrictive_vs_classic}.
In fact, we prove that there are graphs~$G$ for which none of the surrounding cop numbers can be bounded by any function of~$c(G)$ and~$\Delta(G)$.
This proves that the classical game of Cops and Robber is sometimes fundamentally different from all its surrounding versions.

\begin{theorem}
    \label{thm:surrounding_vs_classic}
    There is an infinite family of connected graphs~$G$ with classical cop number $c(G) = 2$ and~$\Delta(G) = 3$ while neither~$c_V(G)$, $c_{V, \mathrm{r}}(G)$, $c_E(G)$ nor~$c_{E, \mathrm{r}}(G)$ can be bounded by any function of~$c(G)$ and the maximum degree~$\Delta(G)$.
\end{theorem}

Additionally, we relate the different surrounding versions to each other.
\Cref{eqn:restrictive_smaller} already gives an upper bound for the cop numbers in the restrictive versions in terms of their corresponding nonrestrictive cop numbers.
To complete the picture, our second contribution is to prove several lower and upper bounds for different combinations:

\begin{theorem}
    \label{thm:bounds}
    Each of the following holds (assuming~$G$ to be connected):
    \begin{enumerate}
        \item\label{itm:v_vr} \makebox[4.5cm][l]{\(
            \forall G : c_V(G) \leq \Delta(G) \cdot c_{V, \mathrm{r}}(G)
        \)} \quad and \quad \(
            \exists G : c_V(G) \geq \Delta(G) \cdot c_{V, \mathrm{r}}(G)
        \)

        \item\label{itm:e_er} \makebox[4.5cm][l]{\(
            \forall G : c_E(G) \leq \Delta(G) \cdot c_{E, \mathrm{r}}(G)
        \)} \quad and \quad \(
            \exists G : c_E(G) \geq \Delta(G)/4 \cdot c_{E, \mathrm{r}}(G)
        \)

        \item\label{itm:v_e} \makebox[4.5cm][l]{\(
            \forall G : c_V(G) \leq 2 \cdot c_E(G)
        \)} \quad and \quad \(
            \exists G : c_V(G) \geq 2 \cdot (c_E(G) - 1)
        \)

        \item\label{itm:vr_er} \makebox[4.5cm][l]{\(
            \forall G : c_{V, \mathrm{r}}(G) \leq 2 \cdot c_{E, \mathrm{r}}(G)
        \)} \quad and \quad \(
            \exists G : c_{V, \mathrm{r}}(G) \geq c_{E, \mathrm{r}}(G)
        \)

        \item\label{itm:e_v} \makebox[4.5cm][l]{\(
            \forall G : c_E(G) \leq \Delta(G) \cdot c_V(G)
        \)} \quad and \quad \(
            \exists G : c_E(G) \geq \Delta(G)/12 \cdot c_V(G)
        \)

        \item\label{itm:er_vr} \makebox[4.5cm][l]{\(
            \forall G : c_{E, \mathrm{r}}(G) \leq \Delta(G) \cdot c_{V, \mathrm{r}}(G)
        \)} \quad and \quad \(
            \exists G : c_{E, \mathrm{r}}(G) \geq \Delta(G)/48 \cdot c_{V, \mathrm{r}}(G)
        \)
    \end{enumerate}
\end{theorem}

Note that all lower and upper bounds from \cref{thm:bounds} are tight up to a small additive or multiplicative constant.
We prove the upper bounds in \cref{sec:upper_bounds}.
The main idea is the same for all six inequalities:
Given a winning strategy for a set of cops in one version, we can simulate the strategy in any other version.
Afterward, in \cref{sec:lower_bounds}, we consider the lower bounds by constructing explicit families of graphs with the desired surrounding cop numbers.
While some lower bounds already follow from standard graph families (like complete bipartite graphs), others need significantly more involved constructions.
For example, we construct a family of graphs from a set of $k-1$ mutually orthogonal Latin squares (see \cref{sec:latin_squares} for a definition).

\subparagraph{Trivial Bounds.}
Clearly, if the robber is surrounded at a vertex~$v_r$, then there must be at least~$\deg(v_r)$ cops around him in all considered versions.
Therefore, the minimum degree~$\delta(G)$ of~$G$ is an obvious lower bound (stated explicitly for~$c_{V, \mathrm{r}}(G)$ in \cite{Burgess2020_Surround}).
Moreover, if the robber can start at a vertex of \emph{highest} degree and \emph{never move}, i.e., in all but the restrictive vertex version, we get the \emph{maximum degree}~$\Delta(G)$ of~$G$ as a lower bound (stated explicitly in~\cite{Crytser2020_Containment} for~$c_{E, \mathrm{r}}(G)$):

\begin{observation}
    \label{obs:trivial_lower_bounds}
    For every connected graph~$G = (V,E)$, we have
    \begin{itemize}
        \item $c_{V, \mathrm{r}}(G) \geq \delta(G)$ as well as
        \item $c_V(G) \geq \Delta(G)$, $c_E(G) \geq \Delta(G)$ and $c_{E, \mathrm{r}}(G) \geq \Delta(G)$.
    \end{itemize}
\end{observation}

\subsection{Related Work}

The game of Cops and Robber was introduced by Nowakowski and Winkler~\cite{Nowakowski1983_VertexToVertex} as well as Quilliot~\cite{Quilliot1978_Jeux} almost forty years ago.
Both consider the case where a single cop chases the robber.
The version with many cops and therefore also the notion of the cop number~$c(G)$ was introduced shortly after by Aigner and Fromme~\cite{Aigner1984_Planar3Cops}, already proving that~$c(G) \leq 3$ for all connected planar graphs~$G$.
Their version is nowadays considered the standard version of the game, and we refer to it as the \emph{classical version} throughout the paper.
The most important open question regarding~$c(G)$ is Meyniel's conjecture, stating that a connected $n$-vertex graph~$G$ has~$c(G) \in O(\sqrt{n})$~\cite{Baird2013_MeynielSurvey,Frankl1987_Meyniel}.
It is known to be \EXPTIME-complete to decide whether~$c(G) \leq k$ (for~$k$ being part of the input)~\cite{Kinnersley2015_EXPTIME}.

By now, countless different versions of the game with their own cop numbers have been considered, see for example these books on the topic~\cite{Bonato2022_Invitation,Bonato2011_TheGameOfCopsAndRobbers}.

The restrictive vertex version was introduced by Burgess et al.~\cite{Burgess2020_Surround}.
They prove bounds for~$c_{V, \mathrm{r}}(G)$ in terms of the clique number~$\omega(G)$, the independence number~$\alpha(G)$ and the treewidth~$\mathrm{tw}(G)$, as well as considering several interesting graph families.
They also show that, for every fixed value of~$k$, it can be decided in polynomial time whether~$c_{V, \mathrm{r}}(G) \leq k$.
The complexity is unknown for~$k$ being part of the input.
Bradshaw and Hosseini~\cite{Bradshaw2019_SurroundingBoundedGenus} extend the study of~$c_{V, \mathrm{r}}(G)$ to graphs of bounded genus, proving (among other results) that~$c_{V, \mathrm{r}}(G) \leq 7$ for every connected planar graph~$G$.
See the bachelor's thesis of Schneider~\cite{Schneider2022_Surrounding} for several further results on~$c_{V, \mathrm{r}}(G)$ (including a version with zugzwang).

The restrictive edge version was introduced even more recently by Crytser, Komarov and Mackey~\cite{Crytser2020_Containment} (under the name \emph{containment variant}).
Besides stating \cref{conj:edge_restrictive_vs_classic}, which is verified for some graphs by \Pralat~\cite{Pralat2015_Containment}, they give several bounds on~$c_{E, \mathrm{r}}(G)$ for different families of graphs.

To the best of our knowledge, $c_V(G)$ and $c_E(G)$ were not considered before.

In light of the (restrictive) vertex and edge versions, one might also define a \emph{face version} for embedded planar graphs.
Here, the cops occupy the faces, and they surround the robber if they occupy all faces incident to~$v_r$.
A restrictive face version could be that the robber must not move along an edge with either one or both incident faces being occupied by a cop.
This version was introduced recently by Ha, Jungeblut and Ueckerdt~\cite{Ha2024_PrimalDual_CGT}.
Despite their similar motivation, the face versions seem to behave differently than the vertex or edge versions.

In each version, one might also add \emph{zugzwang}, i.e., the obligation to actually move during one's turn.
We are not aware of any results about this in the literature.

\subsection{Outline of the Paper}
\Cref{sec:upper_bounds} proves the upper bounds from \cref{thm:bounds}.
Then, in \cref{sec:lower_bounds}, we give constructions implying the corresponding lower bounds.
Finally, in \cref{sec:capturing_is_not_surrounding}, we prove \cref{thm:surrounding_vs_classic}, thereby disproving \cref{conj:edge_restrictive_vs_classic}.

\section{Relating the Different Versions}
\label{sec:upper_bounds}

In this \lcnamecref{sec:upper_bounds}, we prove the upper bounds from \cref{thm:bounds}.
The main idea is always that a sufficiently large group of cops in one version can simulate a single cop in another version.
We denote by~$N_G(v)$ and~$N_G[v]$ the open and closed neighborhood of vertex~$v$ in~$G$, respectively.

\begin{proof}[Proof of \cref{thm:bounds} (Upper Bounds)]
    Let~$G$ be an arbitrary but connected graph.
    \begin{enumerate}
        \item $c_V(G) \leq \Delta(G) \cdot c_{V, \mathrm{r}}(G)$:
        Let~$\mathcal{S}_{V, \mathrm{r}}(G)$ be a winning strategy for~$k \in \mathbb{N}$ restrictive vertex cops~$c_1, \ldots, c_k$ in~$G$.
        For~$i \in \{1, \ldots, k\}$, replace~$c_i$ by a group of~$\Delta(G)$ nonrestrictive vertex cops~$C_i := \{c_i^1, \ldots, c_i^{\Delta(G)}\}$.
        Initially, all cops in~$C_i$ start at the same vertex as~$c_i$ and whenever~$c_i$ moves to an adjacent vertex, all cops in~$C_i$ copy its move.

        If the restrictive cops~$c_1, \ldots, c_k$ arrive in a position where they surround the robber, then he is also surrounded by the groups of cops~$C_1, \ldots, C_k$.
        It remains to consider the case that the robber ends their turn on a vertex~$v$, currently occupied by some group~$C_i$ (a move that would be forbidden in the restrictive version).
        Then the cops in~$C_i$ can spread to the up to~$\Delta(G)$ neighbors of~$v$ in~$G$, thereby surrounding the robber.

        \item $c_E(G) \leq \Delta(G) \cdot c_{E, \mathrm{r}}(G)$:
        Let~$\mathcal{S}_{E, \mathrm{r}}(G)$ be a winning strategy for~$k \in \mathbb{N}$ restrictive edge cops~$c_1, \ldots, c_k$ in~$G$.
        We can replace each cop~$c_i$ by a group of~$\Delta(G)$ nonrestrictive edge cops~$C_i := \{c_1^1, \ldots, c_1^{\Delta(G)}\}$.
        Just as above, the cops in~$C_i$ copy the moves of~$c_i$.
        If the~$c_1, \ldots, c_k$ surround the robber, so do the~$C_1, \ldots, C_k$.
        Further, if the robber moves along an edge occupied by~$C_i$ towards a vertex~$v$ then the cops in~$C_i$ can move to the up to~$\Delta(G)$ incident edges in~$G$, thereby surrounding the robber.

        \item $c_V(G) \leq 2 \cdot c_E(G)$:
        Let~$\mathcal{S}_E(G)$ be a winning strategy for~$k \in \mathbb{N}$ edge cops~$c_1, \ldots, c_k$ in~$G$.
        Each edge cop~$c_i$ occupying an edge~$uv$ can be simulated by two vertex cops~$c_i^1$ and~$c_i^2$ occupying the endpoints~$u$ and~$v$ of~$e$.
        If~$c_i$ moves to an adjacent edge~$e'$, $c_i^1$ and~$c_i^2$ can move to the endpoints of~$e'$ in their next move.
        If the robber is surrounded by some of~$c_1, \ldots, c_k$, then $\bigcup_{i = 1}^k \{c_i^1, c_i^2\} \supseteq N_G[v]$.
        Thus, the robber is surrounded.

        \item $c_{V, \mathrm{r}}(G) \leq 2 \cdot c_{E, \mathrm{r}}(G)$:
        This works exactly as in the previous case.
        However, it remains to verify that two restrictive vertex cops~$c_i^1$ and~$c_i^2$ on the two endpoints of an edge~$e$ also simulate the restrictiveness of the edge cop~$c_i$.
        This is the case, as traversing~$e$ would lead the robber onto a vertex currently occupied by a restrictive vertex cop, which is forbidden.

        \item $c_E(G) \leq \Delta(G) \cdot c_V(G)$:
        Let~$\mathcal{S}_V(G)$ be a winning strategy for~$k \in \mathbb{N}$ vertex cops~$c_1, \ldots, c_k$.
        We replace each vertex cop~$c_i$ by a group of~$\Delta(G)$ edge cops $C_i := \{c_i^1, \ldots, c_i^{\Delta(G)}\}$ that initially position themselves all on an arbitrary edge of~$G$ incident to the vertex occupied by~$c_i$.
        Now, if cop~$c_i$ moves along an edge~$e$, then the cops in~$C_i$ all move to~$e$ in their next move.

        When the vertex cops~$c_1, \ldots, c_k$ surround the robber at a vertex~$v$, then for each neighbor~$u$ of~$v$, there is an edge~$e$ incident to~$u$ occupied by a group~$C_i$.
        However, in general, the edge cops do not yet surround the robber.
        If the robber does not move during its next turn, for each edge~$e$ incident to~$v$, at least one edge cop can move there from an edge~$e'$ adjacent to~$e$, thereby surrounding the robber.
        Otherwise, if the robber moves to a neighbor~$u$ of~$v$, then the~$\Delta(G)$ edge cops on some edge incident to~$u$ can spread to all edges incident to~$u$ in their next turn, thereby surrounding the robber.

        \item $C_{E, \mathrm{r}}(G) \leq \Delta(G) \cdot c_{V, \mathrm{r}}(G)$:
        Let~$\mathcal{S}_{V, \mathrm{r}}$ be a winning strategy for~$k \in \mathbb{N}$ restrictive vertex cops.
        They can be simulated by~$k$ groups of~$\Delta(G)$ restrictive edge cops each, just as in the previous case.
        However, it remains to prove that the robber also loses if he moves onto a vertex~$v$ that would be occupied by a simulated restrictive vertex cop (a move that would be forbidden in~$\mathcal{S}_{V, \mathrm{r}}$).
        In this case, there is a group of~$\Delta(G)$ edge cops on some edge incident to~$v$ and they can spread to all edges incident to~$v$ in their next turn, thereby surrounding the robber.
        \qedhere
    \end{enumerate}
\end{proof}

\begin{corollary}
    For every graph~$G$, the surrounding cop numbers $c_V(G)$, $c_E(G)$, $c_{V, \mathrm{r}}(G)$ and~$c_{E, \mathrm{r}}(G)$ are always within a factor of~$2\Delta(G)$ of each other.
\end{corollary}

\begin{proof}
    In each of the six upper bounds stated in \cref{thm:bounds}, the number of cops increases by at most a factor of~$\Delta(G)$.
    In all cases, this is obtained by simulating a winning strategy of one surrounding variant by (groups of) cops in another variant.
    The only cases where two such simulations need to be combined is when changing both, the cop type (vertex cops/edge cops) and the restrictiveness.
    It is easy to check that in all but one combination the number of cops increases by at most a factor of~$2\Delta(G)$.
    The only exception is when a winning strategy for restrictive vertex cops is simulated by nonrestrictive edge cops, where the number of cops would increase by a factor of~$\Delta(G)^2$.
    However, looking at the proof of \cref{thm:bounds}, we can see that both simulations replace a single cop by a group of~$\Delta(G)$ cops.
    In this particular case, it suffices to do this replacement just once.
\end{proof}

We remark that all upper bounds in \cref{thm:bounds}
result from simulating a winning strategy of another surrounding version.
In the next \lcnamecref{sec:lower_bounds} we show that, surprisingly, these are indeed (asymptotically) tight.
After all, it would seem natural that every version comes with its own unique winning strategy (more involved than just simulating one from a different version).

\section{Explicit Graphs and Constructions}
\label{sec:lower_bounds}

In this \lcnamecref{sec:lower_bounds}, we shall mention or construct several families of graphs with some extremal behavior for their corresponding classical and surrounding cop numbers.
Together, these graphs prove all lower bounds stated in \cref{thm:bounds}.

\subsection{Complete Bipartite Graphs}

We start by considering complete bipartite graphs.
By choosing the right parameters, they already serve to prove two of the lower bounds from \cref{thm:bounds}.
Furthermore, they appear again as a building block for slightly more complicated graphs in \cref{sec:leafy_graphs}.

\begin{proposition}
    \label{prop:complete_bipartite_graphs}
    For a complete bipartite graph~$G$, we have $c(G) = \min\{2, \delta(G)\}$, $c_{V, \mathrm{r}}(G) = \delta(G)$ and $c_V(G) = c_{E, \mathrm{r}}(G) = c_E(G) = \Delta(G)$.
\end{proposition}

\begin{proof}
    In the classical version, it is enough to place one cop in each bipartition class.
    Wherever the robber is, the cop from the other bipartition class can capture him in his next turn, proving~$c(G) \leq 2$.
    Further, if~$\delta(G) = 1$, then a single cop suffices:
    By starting on the single vertex in one bipartition class, the robber is forced to position himself in the other.
    There, he gets captured in the cops' next turn.
    This proves~$c(G) \leq \delta(G)$.
    On the other hand, a single cop is not enough in case that~$\delta(G) \geq 2$:
    The robber can always stay in the same bipartition class as the cop (but on a different vertex).
    We conclude that~$c(G) \geq 2$ in this case, and therefore $c(G) = \min\{2, \delta(G)\}$.

    The lower bounds of~$\delta(G)$ and~$\Delta(G)$ for all surrounding versions follow from \cref{obs:trivial_lower_bounds}.
    Therefore, it remains to show upper bounds below.

    A group of~$\delta(G)$ restrictive vertex cops can initially occupy all vertices of the smaller bipartition class.
    This forces the robber to position himself on a vertex in the larger bipartition class, where he is surrounded immediately.

    A winning strategy for~$\Delta(G)$ nonrestrictive vertex cops is to initially occupy all vertices from the larger bipartition class.
    If the robber positions himself in the smaller bipartition class, he is surrounded immediately, so he has to position himself on one of the cops.
    Then in the cops' next turn, any subset of~$\delta(G)$ cops can occupy the smaller bipartition class, thereby surrounding the robber.

    In the restrictive and nonrestrictive edge versions, we let~$\Delta(G)$ cops position themselves such that, in both bipartition classes, each vertex has at least one incident edge occupied by a cop.
    Then, no matter where the robber positions himself, he gets surrounded by the cops within their next turn.
\end{proof}

Let us consider two special cases of \cref{prop:complete_bipartite_graphs} for all~$\Delta \in \mathbb{N}$:
First, the star~$K_{\Delta,1}$ has $c_{V, \mathrm{r}}(K_{\Delta,1}) = 1$ while $c_V(K_{\Delta,1}) = \Delta$, thus proving the lower bound in \cref{itm:v_vr} of \cref{thm:bounds}.
Second, the complete bipartite graph~$K_{\Delta,\Delta}$ has~$c_{V, \mathrm{r}}(K_{\Delta,\Delta}) = c_{E, \mathrm{r}}(K_{\Delta,\Delta}) = \Delta$, thus proving the lower bound in \cref{itm:vr_er} of \cref{thm:bounds}.

\subsection{Regular Graphs with Leaves}
\label{sec:leafy_graphs}

Our first construction takes a connected $k$-regular graph~$H$ and attaches a set of~$\ell$ new degree-$1$-vertices (\emph{leaves}) to each vertex.
Depending on~$H$,~$k$ and~$\ell$ we can give several bounds on the surrounding cop numbers of the resulting graph.

\begin{lemma}
    \label{lem:leafy_graphs}
    Let~$H = (V_H, E_H)$ be a $k$-regular connected graph and let $G = (V_G, E_G)$ be the graph obtained from~$H$ by attaching to each vertex~$v \in V_H$ a set of~$\ell$ new leaves for some~$\ell \geq 0$.
    Then each of the following holds:
    \begin{enumerate}
        \item\label{itm:lower_bound_leafy_vertex} \(
            c_V(G) \geq
            \begin{cases*}
                k(k+\ell-1) & if $\girth(H) \geq 7$ \\
                (k+1)\ell & always
            \end{cases*}
        \)
        \item $c_{V, \mathrm{r}}(G) = \max\{c_{V, \mathrm{r}}(H), k + 1\}$
        \item \(
            c_E(G) \geq
            \begin{cases*}
                k(k+\ell-1) & if $\girth(H) \geq 6$ \\
                k\ell & if $\girth(H) \geq 4$ \\
                \frac{1}{2} (k+1)\ell & always
            \end{cases*}
        \)
        \item $c_{E, \mathrm{r}}(G) = \max\{c_{E, \mathrm{r}}(H), k+\ell\}$
    \end{enumerate}
\end{lemma}

\begin{proof}
    Note that most claimed inequalities hold trivially for the case that~$\ell = 0$ (many lower bounds become~$0$, while others follow from~$G = H$ in this case).
    Only the two cases requiring~$\girth(H) \geq 6$, respectively~$\girth(H) \geq 7$, are not directly clear.
    However, their proofs below hold for~$\ell = 0$ as well.
    In all other cases, we implicitly assume~$\ell \geq 1$ to avoid having to handle additional corner cases.

    \begin{enumerate}
        \item To prove the lower bounds on~$c_V(G)$, consider any configuration of cops on the vertices of~$G$.
        For a vertex $v \in V_H$ of~$G$, let~$A_v$ be the set consisting of~$v$ and all leaves that are attached to it, i.e., $A_v = \{v\} \cup (N_G[v] \setminus V_H)$.
        We call a vertex $v \in V_H$ \emph{safe} if there are fewer than~$\ell$ cops on~$A_v$ in~$G$.
        Note that if the robber ends his turn on a safe vertex, then the cops cannot surround him in their next turn.
        Let~$v_r$ be the current position of the robber.
        If the total number of cops is less than~$(k+1)\ell$, then at least one of the~$k+1$ vertices in the closed neighborhood~$N_H[v_r]$ of~$v_r$ is safe, as no cop can be in~$A_v$ and~$A_w$ for~$v \neq w$.
        Thus, the robber always has a safe vertex to move to (or to remain on), giving him a strategy to avoid being surrounded.
        It follows that $c_V(G) \geq (k+1)\ell$.

        Now, if~$\girth(H) \geq 7$ and the robber is on~$v_r \in V_H$, then we consider for each neighbor~$v$ of~$v_r$ in~$V_H$ additionally the set~$B_v = N_G[N_G(v) \setminus \{v_r\}]$, i.e., all vertices~$w$ with~$\dist(w, v) \leq 2$ except from~$N_G[v_r] \setminus \{v\}$.
        Since $\girth(H) \geq 7$, we have that~$B_v \cap B_w = \emptyset$ for distinct~$v, w \in N_H(v_r)$.
        Similar to above, we call~$v \in N_H(v_r)$ \emph{safe} if~$B_v$ contains fewer than~$k+\ell-1$ cops.
        Again, if the robber ends his turn on a safe vertex, the cops cannot surround him in their next turn.
        If the total number of cops is less than~$k(k+\ell-1)$, then at least one of the~$k$ neighbors of~$v_r$ in~$H$ is safe.
        This would give the robber a strategy to avoid being surrounded.
        It follows that~$c_V(G) \geq k(k+\ell-1)$ in the case that~$\girth(H) \geq 7$.

        \item To prove that $c_{V, \mathrm{r}}(G) \geq \max\{c_{V, \mathrm{r}}(H), k + 1\}$, first note that $c_{V, \mathrm{r}}(G) \geq c_{V, \mathrm{r}}(H)$ because the leaves do not help the cops.
        Further, $c_{V, \mathrm{r}}(G) \geq k + 1$ holds because the robber versus~$k$ cops can easily always stay on the subgraph~$H$ of~$G$ where it is impossible to surround him (in~$G$) with~$k$ cops only (note that~$G$ has at least~$k + 1$ vertices).

        For the upper bound $c_{V, \mathrm{r}}(G) \leq \max\{c_{V, \mathrm{r}}(H), k+1\}$, we describe a two-phase winning strategy for the cops.
        In the first phase, let~$c_{V, \mathrm{r}}(H)$ cops execute an optimal strategy for surrounding the robber on the subgraph~$H$ of~$G$.
        If during that phase the robber leaves~$H$, i.e., steps on a leaf attached to some $v \in V_H$, act as if the robber stays on~$v$.
        Hence, after finitely many rounds, the robber is on a vertex in the set $N_v = \{v\} \cup (N_G(v) - V_H)$ for some~$v \in V_H$, and~$k$ of the~$c_{V, \mathrm{r}}(H)$ cops occupy all~$k$ neighbors of~$v$ in~$V_H$.
        As we are in the restrictive variant, the robber may not leave~$N_v$.
        Now, in the second phase, let one of the remaining $\max\{c_{V, \mathrm{r}}(H), k+1\} - k \geq 1$ cops go to vertex~$v$.
        This forces the robber to go to an attached leaf and be surrounded there.

        \item To prove the three lower bounds on~$c_E(G)$, consider any configuration of cops on the edges of~$G$.
        For a vertex $v \in V_H$ of~$G$, let~$E_v$ be the set of all edges in~$G$ incident to~$v$.
        We call vertex $v \in V_H$ \emph{safe} if there are fewer than~$\ell$ cops in total on~$E_v$ in~$G$.
        Note that if the robber ends his turn on a safe vertex, then the cops can not surround him in their next turn.
        Let the robber be at vertex~$v_r \in V_H$.
        If all vertices~$v \in N_H[v_r]$ were unsafe, then each such~$v$ has at least~$\ell$ cops in~$E_v$.
        Further, each cop is on an edge incident to at most two vertices in~$N_H[v_r]$.
        Thus, at least $\frac{1}{2}(k+1)\ell$ cops are required to make all vertices in~$N_H[v_r]$ unsafe, i.e., for fewer than $\frac{1}{2}(k+1)\ell$ cops there is always a safe vertex for the robber to move onto.

        If we further assume~$\girth(H) \geq 4$, then no two vertices in~$N_H(v_r)$ are connected by an edge, so the~$k$ vertices in~$N_H(v_r)$ each require their own~$\ell$ cops on incident edges to be unsafe.
        In other words, fewer than~$k\ell$ cops always guarantee the robber a safe vertex to move onto.

        Now we even assume that~$\girth(G) \geq 6$.
        Similar to the vertex version, we define for each~$v \in N_H(v_r)$ a set~$B_v$ containing all edges incident to a vertex in~$N_G[v]$ that do not have~$v_r$ as an endpoint.
        We say that~$v$ is \emph{safe} if there are at most~$k+\ell-1$ cops in~$B_v$.
        From~$\girth(H) \geq 6$ it follows that~$B_v \cap B_w = \emptyset$ for distinct neighbors~$v, w$ of~$v_r$.
        Thus, for all~$v \in N_H(v_r)$ to be unsafe, we would need at least~$k(k+\ell-1)$ cops.

        \item To prove that~$c_{E, \mathrm{r}}(G) \geq \max\{c_{E, \mathrm{r}}(H), k + \ell\}$, first note that~$c_{E, \mathrm{r}}(G) \geq c_{E, \mathrm{r}}(H)$ because the leaves do not help the cops.
        Further, $c_{E, \mathrm{r}}(G) \geq k + \ell = \Delta(G)$ holds by \cref{obs:trivial_lower_bounds}.

        Finally, the upper bound $c_{E, \mathrm{r}}(G) \leq \max\{c_{E, \mathrm{r}}(H), k+\ell\}$ for the restrictive edge version follows mostly along the lines of the restrictive vertex version above.
        In a first phase, we use $c_{E, \mathrm{r}}(H)$ cops to force the robber onto a vertex~$v$ or a leaf attached at~$v$, while each of the~$k$ incident edges at~$v$ in~$E_H$ is occupied by a cop.
        In the second phase, let~$\ell$ of the remaining~$\max\{c_{E, \mathrm{r}}(H), k+\ell\} - k \geq \ell$ cops go to occupy the remaining~$\ell$ incident edges at~$v$ (the ones that are not in~$E_H$).
        This surrounds the robber, either on~$v$ or an attached leaf.
        \qedhere
    \end{enumerate}
\end{proof}

Applied to different host graphs~$H$, \cref{lem:leafy_graphs} yields several interesting bounds, summarized in the following corollaries:
\Cref{cor:leafy_complete_bipartite} proves the lower bound in \cref{itm:e_er} for even~$\Delta$, and \cref{cor:leafy_edge} proves the lower bound in \cref{itm:v_e}.

\begin{corollary}
    \label{cor:leafy_complete_bipartite}
    For every~$\Delta \geq 2$ there is a connected graph~$G$ with~$\Delta(G) = \Delta$ such that \(
        c_{V, \mathrm{r}}(G) =
        \bigl\lfloor \frac{\Delta}{2} \bigr\rfloor + 1
    \),  \(
        c_V(G) =
        \bigl(\bigl\lfloor \frac{\Delta}{2} \bigr\rfloor + 1\bigr)
        \bigl\lceil \frac{\Delta}{2} \bigr\rceil
    \), \(
        c_{E, \mathrm{r}}(G) = \Delta(G)
    \) and \(
        c_E(G) =
        \bigl\lfloor \frac{\Delta}{2} \bigr\rfloor
        \bigl\lceil \frac{\Delta}{2} \bigr\rceil
    \).
\end{corollary}

\begin{proof}
    Let~$k = \bigl\lfloor \frac{\Delta}{2} \bigr\rfloor$, let~$H = K_{k,k}$ be the $k$-regular complete bipartite graph and choose $\ell = \bigl\lceil \frac{\Delta}{2} \bigr\rceil$.
    We consider the graph~$G$ obtained from~$H$ by attaching~$\ell$ new leaves to each vertex.
    Note that~$G$ has maximum degree~$\Delta(G) = k + \ell = \Delta$.
    \begin{itemize}
        \item \Cref{lem:leafy_graphs} yields that~$c_{V, \mathrm{r}}(G) = \max\{c_{V, \mathrm{r}}(H), k + 1\}$.
        From \cref{prop:complete_bipartite_graphs}, we further get that~$c_{V, \mathrm{r}}(H) = k$, so actually~$c_{V, \mathrm{r}}(G) = k + 1 = \bigl\lfloor \frac{\Delta}{2} \bigr\rfloor + 1$.

        \item \Cref{lem:leafy_graphs} gives~$c_V(G) \geq (k + 1)\ell$.
        It remains to show that $(k+1)\ell$ nonrestrictive vertex cops can always surround the robber, i.e., that $c_V(G) \leq (k+1)\ell$.
        We give a winning strategy for~$(k+1)\ell$ cops.
        Let~$A$ and~$B$ be the two bipartition classes of~$H$.
        Initially, we place~$\ell$ cops on each of the~$k$ vertices in~$A$ and the remaining~$\ell$ cops such that each~$b \in B$ is occupied by at least one cop (this is always possible as~$\ell \geq k$).
        Now, if the robber starts at a leaf, he is surrounded immediately.
        If he starts at~$a \in A$, then the~$\ell$ cops on~$a$ spread across all adjacent leaves, thereby surrounding the robber.
        Thus, the robber has to start at~$b \in B$.
        If he stays on~$b$ or an adjacent leaf, he can be surrounded by the~$\ell$ cops currently on~$B$ in a finite number of turns.
        Thus, he has to move to~$a \in A$ at some point.
        Then~$\ell$ of the cops on~$a$ move to the adjacent leaves and all other cops currently on some vertices in~$A$ move to~$B$ in a way such that all vertices of~$B$ are occupied.
        (If~$k \geq 2$, the~$\ell$ cops on some~$a' \in A \setminus \{a\}$ can already occupy all~$b \in B$.
        Otherwise, if~$k = 1$, then the~$\ell$ cops not on leaves adjacent to~$a$ can always reach the unique~$b \in B$.)
        Then the robber is surrounded.

        \item By \cref{lem:leafy_graphs}, we get that~$c_{E, \mathrm{r}}(G) = \max\{c_{E, \mathrm{r}}(H), k + \ell\}$, and it follows from \cref{prop:complete_bipartite_graphs} that~$c_{E, \mathrm{r}}(H) = k + \ell$.
        Thus, $c_{E, \mathrm{r}}(G) = k + \ell = \Delta$ holds.

        \item Considering~$c_E(G)$, we apply \cref{lem:leafy_graphs} and get~$c_E(G) \geq k\ell = \bigl\lfloor \frac{\Delta}{2} \bigr\rfloor \bigl\lceil \frac{\Delta}{2} \bigr\rceil$ (using that $\girth(H) \geq 4$, as~$H$ is a bipartite graph).
        It remains to show that~$k\ell$ nonrestrictive edge cops suffice.
        To this end, let~$k^2$ cops initially occupy all edges of~$H$, one cop per edge.
        For odd~$\Delta$, let~$M$ be a perfect matching in~$H$ and let each of the~$k$ remaining cops occupy a single edge from~$M$.
        Now if the robber starts on a leaf~$w$ adjacent to a vertex~$v$ of~$H$, he gets surrounded by one cop moving from an edge of~$H$ incident to~$v$ onto the edge incident to~$w$.
        Otherwise, if the robber starts on some vertex~$v$ of~$H$, the~$k$ edges of~$H$ incident to~$v$ contain~$\ell$ cops, which can spread across the~$\ell$ leaves incident to~$v$.
        If~$k = 1$, there is one remaining cop which can move to the only edge of~$H$, thereby surrounding the robber.
        In the remaining cases that~$k \geq 2$, each of the~$k$ neighbors of~$v$ in~$H$ is incident to at least one other edge of~$H$ with a cop on.
        This cop can move onto the edge incident to~$v$, thereby surrounding the robber.
        \qedhere
    \end{itemize}
\end{proof}

\begin{corollary}
    \label{cor:leafy_edge}
    For every~$\Delta \geq 2$, there is a connected graph~$G$ with~$\Delta(G) = \Delta$ such that~$c_V(G) = 2(\Delta - 1)$ and~$c_E(G) = \Delta$.
\end{corollary}

\begin{proof}
    Let~$H = K_2$ be a single edge~$uv$, so~$k = 1$, and choose~$\ell = \Delta - 1$.
    Then it holds that~$\Delta(G) = \Delta$.

    \begin{itemize}
        \item It follows from \cref{lem:leafy_graphs} that~$c_V(G) \geq (k+1)\ell = 2(\Delta - 1)$.
        On the other hand, it is easy to see that~$2(\Delta - 1)$ vertex cops suffice:
        Initially, $\Delta - 1$ cops position themselves at~$u$ and~$v$ each.
        If the robber positions himself on a leaf, he is surrounded immediately.
        If he starts on~$v$ (the case for~$u$ is symmetric), then the cops at~$v$ move to the~$\Delta-1$ leaves, thereby surrounding the robber (together with the cops at~$u$).

        \item To prove that~$c_E(G) = \Delta$, first observe that~$c_E(G) \geq \Delta$ by \cref{obs:trivial_lower_bounds} (\cref{lem:leafy_graphs} gives a worse lower bound in this case).
        A winning strategy for~$\Delta$ edge cops goes as follows:
        All cops start on edge~$uv$.
        If the robber starts on~$v$ (or symmetrically~$u$), then one of the cops stays on~$uv$, while the other~$\Delta - 1$ occupy the edges incident to the leaves adjacent to~$v$.
        If the robber starts on a leaf~$w$, a single cop can move onto the edge incident to~$w$.
        \qedhere
    \end{itemize}
\end{proof}

\subsection{Graphs from Mutually Orthogonal Latin Squares}
\label{sec:latin_squares}

A \emph{Latin square} of order $k \geq 1$ is a $(k \times k)$-array filled with numbers from $[k] = \{1, \ldots, k\}$ such that each row and each column contains each number from~$[k]$ exactly once.
Formally, a Latin square~$L$ is a partition $L(1) \cup \cdots \cup L(k)$ of $A = [k] \times [k]$ such that row~$i$ (with $i \in [k]$) is $A[i, \cdot] = \{(i,j) \in A \mid j \in [k]\}$, and for every number $n \in [k]$ we have $\abs{A[i, \cdot] \cap L(n)} = 1$, and symmetrically for the columns.
See the left of~\cref{fig:merged-MOLS} for two different Latin squares.

Let~$L_1$ and~$L_2$ be two Latin squares of order~$k$.
Their \emph{juxtaposition} $L_1 \otimes L_2$ is the Latin square of order~$k$ that contains in each cell the ordered pair of the entries of~$L_1$ and~$L_2$ in that cell.
We say that~$L_1$ and~$L_2$ are \emph{orthogonal} if each ordered pair appears exactly once in~$L_1 \otimes L_2$, i.e., if for every two distinct $n_1,n_2 \in [k]$ we have $\abs{L_1(n_1) \cap L_2(n_2)} = 1$.
It is well known that~$k-1$ \emph{mutually orthogonal Latin squares (MOLS)} $L_1, \ldots, L_{k-1}$ (meaning that $L_s$ and~$L_t$ are orthogonal whenever~$s \neq t$) exist if and only if~$k$ is a prime power~\cite{Keedwell2015_LatinSquares}.
The two Latin squares in \cref{fig:merged-MOLS} (left) are indeed orthogonal, as can be seen by their juxtaposition below.

\medskip
\noindent
Let~$k$ be a prime power and $L_1, \ldots, L_{k-1}$ a set of~$k-1$ mutually orthogonal Latin squares of order~$k$.
We construct a graph~$G_k$.
Let $A = [k] \times [k]$ denote the set of all positions, $R = \{ A[i,\cdot] \mid i \in [k]\}$ denote the set of all rows in~$A$, and $\mathcal{L} = \{L_s(n) \mid s \in [k-1] \land n \in [k]\}$ denote the set of all parts of the Latin squares $L_1, \ldots, L_{k-1}$.
Then, $G_k = (V,E)$ is the graph with
\[
    V = A \cup R \cup \mathcal{L} \quad \text{and}
    \quad E = \{pS \mid p \in A,~S \in R \cup \mathcal{L},~p \in S\}
    \text{.}
\]
We observe that~$G_k$ is a $k$-regular bipartite graph with $\abs{A} + \abs{R \cup \mathcal{L}} = k^2+(k + k(k-1)) = 2k^2$ vertices.
It has an edge between position $p \in A$ and a set $S \in R \cup \mathcal{L}$ if and only if~$p$ is in set~$S$.
See also the right of~\cref{fig:merged-MOLS} for a schematic drawing.

\begin{figure}[tb]
    \centering
    \includegraphics{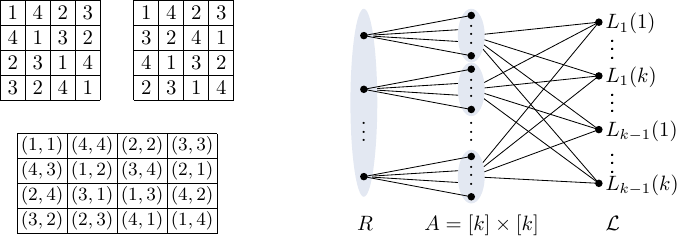}
    \caption{
        Left:
        Two Latin squares and their juxtaposition, proving that they are orthogonal.
        Right:
        The graph~$G_k$ created from~$k-1$ MOLS of order~$k$.
        The vertices in~$R$ correspond to the rows of~$A$, the middle vertices correspond to the cells of~$A$ (ordered row by row in the drawing) and the vertices in~$\mathcal{L}$ correspond to the parts of the MOLS.
        }
    \label{fig:merged-MOLS}
\end{figure}

\begin{lemma}
    \label{lem:MOLS_girth_6}
    For a prime power~$k$, graph~$G_k$ has~$\girth(G_k) \geq 6$.
\end{lemma}

\begin{proof}
    It suffices to show that~$G_k$ has no $4$-cycle, since~$G_k$ is bipartite (and therefore has no odd cycles).
    First note that every $4$-cycle must be of the form $(p_1, S_1, p_2, S_2)$, in particular, it must contain two vertices~$p_1$ and~$p_2$ from~$A$.
    However, there is neither a $4$-cycle
    \begin{itemize}
        \item with $S_1, S_2 \in R$, since every position is in at most one row,
        \item nor with $S_1 \in R$ and $S_2 \in \mathcal{L}$, since every part~$L_s(n)$ of a Latin square contains only one position of each row,
        \item nor with $S_1, S_2 \in \mathcal{L}$ being parts of the same Latin square~$L_s$, since every position is in only one part of~$L_s$,
        \item nor with $S_1, S_2 \in \mathcal{L}$ being parts of different Latin squares~$L_s$ and~$L_t$, since~$L_s$ and~$L_t$ are orthogonal (and hence no two distinct positions are in the same parts in~$L_s$ and in~$L_t$).
        \qedhere
    \end{itemize}
\end{proof}

\begin{lemma}
    \label{lem:MOLS}
    For a prime power~$k$, graph~$G_k$ has $c(G_k) = k$, $c_{V, \mathrm{r}}(G_k) \leq k+1$ and~$c_E(G_k) \geq k(k-1)$.
\end{lemma}

\begin{proof}
    \begin{itemize}
        \item To prove~$c(G_k) \geq k$, we use a theorem from Aigner and Fromme~\cite{Aigner1984_Planar3Cops} stating that any graph with girth at least~$5$ has~$c(G) \geq \delta(G)$.
        As~$G_k$ is $k$-regular and has $\girth(G_k) \geq 6$ (\cref{lem:MOLS_girth_6}), we get~$c(G_k) \geq k$.

        On the other hand,~$k$ cops suffice:
        Initially, the cops occupy the~$k$ vertices in~$R$, so the robber has to position himself in~$A$ or~$\mathcal{L}$.
        If he starts at a vertex~$v \in A$, then this vertex must be in some row and the cop on the corresponding vertex in~$R$ can capture him in his next move.
        Otherwise, if the robber starts on some~$L_i(n) \in \mathcal{L}$ (for~$i \in [k-1]$ and~$n \in [k]$), then~$L_i(n)$ has exactly~$k$ neighbors in~$A$, one per row.
        The cops can move to~$N_{G_k}(L_i(n))$, thereby surrounding the robber (and thereby capturing him in their next move).

        \item To show that~$k+1$ restrictive vertex cops suffice to surround the robber, we start by placing one cop on each of the~$k$ vertices in~$R$.
        Thus, the robber may not enter~$R$.
        Now, if the robber ends his turn on a vertex in $\mathcal{L}$ corresponding to the part~$L_s(n)$ of the Latin square~$L_s$ for some $n \in [k]$, then the~$k$ cops on~$R$ can move to the vertices in~$A$ corresponding to the positions in~$L_s(n)$, and thus surround the robber.
        This works because the elements in~$L_s(n)$ are from all~$k$ different rows (because~$L_s$ is a Latin square) and there is one cop responsible for every row.
        Thus, we may assume that the robber starts on a vertex in $A$ and never moves.
        But then, it is enough to use the $(k+1)$-th cop to go to the robber vertex and force him to move (recall that we are in the restrictive version).
        The only possible move is to a vertex from~$\mathcal{L}$ where he gets surrounded in the next move of the cops as seen above.

        \item The lower bound for nonrestrictive edge cops follows from~$G_k$ being $k$-regular and having~$\girth(G) \geq 6$ by \cref{lem:MOLS_girth_6}.
        Attaching~$\ell = 0$ leaves to each vertex, \cref{lem:leafy_graphs} yields~$c_E(G) \geq k(k+\ell-1) = k(k-1)$.
        \qedhere
    \end{itemize}
\end{proof}

\begin{remark}
    Burgess et al.~\cite{Burgess2020_Surround} notice that, for graphs~$G$ of many families with a \enquote{large} value of~$c_{V, \mathrm{r}}(G)$, the classical cop number~$c(G)$ was \enquote{low} (often even constant).
    In fact, they only provide a single family of graphs (constructed from finite projective planes) where~$c(G) \approx c_{V, \mathrm{r}}(G)$.
    They ask (Question~$7$ in~\cite{Burgess2020_Surround}) whether graphs with large~$c_{V, \mathrm{r}}(G)$ inherently possess some property that implies that~$c(G)$ is low.
    Our graph $G_k$ from $k-1$ MOLS satisfies $c(G_k),c_{V,\mathrm{r}}(G_k) \in \{k,k+1\}$.
    We interpret this as evidence that there is no such property.
\end{remark}

\subsection{Line Graphs of Complete Graphs}

The \emph{line graph}~$L(G)$ of a given graph~$G=(V,E)$ is the graph whose vertex set consists of the edges~$E$ of~$G$, and two vertices of~$L(G)$ are connected if their corresponding edges in~$G$ share an endpoint.
For~$n \geq 3$, let~$K_n$ denote the complete graph on the set~$[n] = \{1, \ldots, n\}$.
For distinct $x,y \in [n]$, we denote by~$\{x,y\}$ the vertex of $L(K_n)$ corresponding to the edge between~$x$ and~$y$ in~$K_n$.
Burgess et al.~\cite{Burgess2020_Surround} showed that $c_{V, \mathrm{r}}(L(K_n)) = 2(n-2) = \delta(L(K_n))$.
This is obtained by placing the cops on all vertices~$\{1,x\}$ for~$x \in \{2, \ldots n\}$ and~$\{2,y\}$ for $y \in \{3, \ldots, n - 1\}$.
The cops can surround the robber in their first move.

\begin{lemma}
    \label{lem:line_graph_of_complete_graph}
    For every $n \geq 3$ we have $c_V(L(K_n)) = 2(n-2)$, $c_E(L(K_n)) \geq n(n-2)/3$, and $c_{E, \mathrm{r}}(L(K_n)) \geq (n^2 - 4)/12$.
\end{lemma}

\begin{proof}
    As $2(n-2) = \Delta(L(K_n)) \leq c_V(L(K_n))$ (\cref{obs:trivial_lower_bounds}), it is enough to observe that the strategy for the vertex cops provided by Burgess et al.~\cite{Burgess2020_Surround} also works for vertex cops in the nonrestrictive version.
    In their strategy, the cops start at the vertices $\{1,x\}$ for~$x \in \{2, \ldots, n\}$ and~$\{2,y\}$ for~$y \in \{3, \ldots, n - 1\}$ of~$L(K_n)$.
    They prove that these cops can surround the robber immediately after he chose his start vertex, in particular, before the robber first moves along an edge.
    Therefore, the only remaining cases to check for the nonrestrictive version are what happens if the robber starts on a cop (which would be forbidden in the restrictive setting):
    \begin{itemize}
        \item If the robber starts on vertex~$\{1,2\}$, then he is surrounded as soon as the cop from~$\{1,2\}$ moves to~$\{2,n\}$.

        \item If the robber starts on a vertex~$\{1,x\}$, we first consider the case that~$x \in \{3, \ldots, n - 1\}$.
        The robber gets surrounded by the cop at~$\{1,x\}$ moving to~$\{x,n\}$, and each cop from~$\{2,y\}$ for~$y \in \{3, n-1\} \setminus \{x\}$ moving to~$\{x,y\}$.
        In the case that~$x = n$, then the cop on~$\{1,2\}$ moves to~$\{n,2\}$, and each cop on~$\{2,y\}$ for~$y \in \{3, \ldots, n-1\}$ moves to~$\{n,y\}$.

        \item Lastly, if the robber starts on a vertex~$\{2,y\}$ for~$y \in \{3, \ldots, n-1\}$, then the cop on~$\{2,y\}$ moves to~$\{1,y\}$, and each cop on~$\{1,x\}$ for~$x \in \{3, \ldots, n\}$ moves to~$\{y,x\}$.
    \end{itemize}

    We now consider the lower bounds on~$c_{E}(L(K_n))$ and~$c_{E, \mathrm{r}}(L(K_n))$, and describe an evasion strategy for the robber, provided that the number~$k$ of edge cops is sufficiently small (an exact bound will  be determined later for both versions).
    To this end, recall that each edge cop occupies an edge of~$L(K_n)$ that corresponds to a copy of~$P_3$ (the path on two edges) in~$K_n$.
    Each such~$P_3$ has a \emph{midpoint} and two \emph{leaves}, each being a vertex of~$K_n$.
    An edge cop moving one step from~$e_1$ to~$e_2$ in~$L(K_n)$ corresponds to~$P_3^1$ and~$P_3^2$ in~$K_n$.
    Here~$P_3^1$ shares an edge with~$P_3^2$, and thus, the midpoint of~$P_3^2$ is contained (as a midpoint or leaf) in~$P_3^1$.
    At the moment that the robber gets surrounded on a vertex~$\{u,v\}$ of~$L(K_n)$, the edge cops occupy all~$n-2$ edges of~$L(K_n)$ corresponding to~$P_3$'s in~$K_n$ with midpoint~$u$ and leaf~$v$, as well as all~$n-2$ edges of~$L(K_n)$ corresponding to~$P_3$'s in~$K_n$ with midpoint~$v$ and leaf~$u$.
    Thus, in the step immediately preceding the surround, each of~$u$ and~$v$ was contained in the~$P_3$ of at least~$n-2$ edge cops.
    
    With this in mind, for any vertex~$v$ of~$K_n$ and point in time~$t \in \N$, let~$p^t(v)$ be the number of edge cops whose corresponding~$P_3$ contains~$v$ before the~$t$-th turn of the robber.
    Similarly, let~$q^t(v)$ be the number of edge cops whose corresponding~$P_3$ contains~$v$ as a midpoint.
    For all~$v \in [n]$ and~$t \geq 2$, it holds that
    \begin{equation}
        \label{eq:p_vs_q}
        q^t(v) \leq p^{t-1}(v)
        \text{.}
    \end{equation}
    Further, recall that each cop has exactly one corresponding~$P_3$, so we have
    \begin{equation}
        \label{eq:average_p(v)}
        \sum_{v \in [n]} p^t(v) = 3k
    \end{equation}
    for all~$t \in \N$.
    The robber cannot be surrounded on a vertex~$\{u,v\}$ of~$L(K_n)$ (with~$u,v \in [n]$, $u \neq v$) if~$\min\{p^t(u), p^t(v)\} < n - 2$.
    Therefore, we call a vertex~$w \in [n]$ of~$K_n$ \emph{safe (at time~$t$)} if~$p^t(w) < n - 2$.

    \begin{itemize}
        \item We start with the nonrestrictive edge version, and choose~$k < n(n-2)/3$.
        Then, Equation~\eqref{eq:average_p(v)} implies that
        \[
            \#\left\{w \in [n] \mid p^t(w) \geq n-2\right\}
            \leq \frac{3k}{n-2}
            < \frac{3n(n-2)}{3(n-2)}
            = n
        \]
        for every~$t \in \N$.
        In particular, there is always a safe vertex~$w \in [n]$.
        For~$t = 1$, i.e., in his first turn, the robber can go to~$\{v,w\}$ with an arbitrary~$v \in [n] \setminus \{w\}$.
        For~$t > 1$, let~$\{u,v\}$ be the robber's current position.
        Then, as~$K_n$ is a complete graph, the robber can move to~$\{v,w\}$ in a single step.
        In both cases, the robber avoids being surrounded because~$p^t(v,w) \leq p^t(w) < n - 2$.

        \item We now consider the restrictive edge version.
        We prove that the robber can avoid being surrounded by ending his $t$-th turn (for~$t \in \N$) on a vertex~$\{u,v\}$ of~$L(K_n)$ with~$p^t(u,v) < (n-2)/2$.
        It follows from~$n \geq 3$ that~$(n-2)/2 < n-2$, so at least one of~$u$ or~$v$ is safe at time~$t$.

        Let~$k < (n^2-4)/12$.
        For~$t = 1$, i.e., the robber's first move, it suffices to identify a single~$u \in [n]$ with~$p^1(u) < (n-2)/2$.
        Then, the robber can start at~$\{u,v\}$ with an arbitrary~$v \in [n] \setminus \{u\}$.
        Such a~$u$ always exists because
        \begin{equation}
            \label{eq:counting_safe_vertices}
            \#\left\{w \in [n] \mid p^1(w) \geq \frac{n-2}{2}\right\}
            \leq \frac{3k}{(n-2)/2}
            = \frac{6k}{n-2}
            < \frac{n+2}{2}
            < n
            \text{.}
        \end{equation}
        Here, the first inequality follows from Equation~\eqref{eq:average_p(v)}, the second from the choice of~$k$, and the third from~$n \geq 3$.

        \medskip
        For~$t > 1$, let~$\{u,v\}$ be the current robber position, chosen in turn~$t-1$ with~$p^{t-1}(u,v) < (n-2)/2$.
        We assume without loss of generality that~$p^{t-1}(u) < (n-2)/2$, and therefore~$q^t(u) < (n-2)/2$ (see Equation~\eqref{eq:p_vs_q}).
        Then, less than~$(n-2)/2$ cops correspond to~$P_3$'s in~$K_n$ with midpoint~$u$.
        This means that less than~$(n-2)/2$ of the~$n-2$ edges of the form~$\bigl\{\{u,v\}, \{u,w\}\bigr\}$ (with~$w \in [n] \setminus \{u,v\}$) are blocked by a cop.
        In particular, more than~$(n-2)/2$ of these edges can be taken by the robber to move to a vertex~$\{u,w\}$ of~$L(K_n)$.

        It remains to prove that at least one of these~$w \in [n]$ is safe, i.e., has~$p^t(w) < (n-2)/2$.
        This is the case because, by Equation~\eqref{eq:counting_safe_vertices},
        there are more than~$n - (n+2)/2 = (n-2)/2$ vertices~$w \in [n]$ with~$p^t(w) < (n-2)/2$.
        If~$u$ or~$v$ are among these, i.e., $\min\{p^t(u), p^t(v)\} < (n-2)/2$, then the robber stays at his current position.
        Otherwise, more than half of the~$w \in [n] \setminus \{u,v\}$ are safe.
        At the same time, more than half of the edges $\bigl\{\{u,v\}, \{u,w\}\bigr\}$ are free for the robber.
        We conclude, that there is at least one~$w^*$ with~$p^t(w^*) < (n-2)/2$ such that the robber can move to~$\{u,w^*\}$.
        \qedhere
    \end{itemize}
\end{proof}

The bounds from \cref{lem:line_graph_of_complete_graph} are stated in terms of the number of vertices~$n$.
Stating them in terms of their maximum degree~$\Delta \coloneqq \Delta(L(K_n)) = 2(n-2)$, we obtain the claimed lower bounds in \cref{itm:e_v,itm:er_vr} of \cref{thm:bounds}.
\begin{align*}
    c_V(L(K_n)) &= \Delta &
    c_E(L(K_n)) &\geq \frac{\Delta^2 + 4\Delta}{12} \\
    c_{V, \mathrm{r}}(L(K_n)) &= \Delta & 
    c_{E, \mathrm{r}}(L(K_n)) &\geq \frac{\Delta^2 + 8\Delta}{48}
\end{align*}
With these, all lower bounds from \cref{thm:bounds} are proven.

\section{When Capturing is not Surrounding}
\label{sec:capturing_is_not_surrounding}

This section is devoted to the proof of \cref{thm:surrounding_vs_classic}, i.e., that none of the four surrounding cop numbers can be bounded by any function of the classical cop number and the maximum degree of the graph.
In particular, we shall construct, for infinitely many integers~$k \geq 1$, a graph~$G_k$ with~$c(G_k) = 2$ and~$\Delta(G_k) = 3$, but~$c_{V, \mathrm{r}}(G_k) \geq k$.
\Cref{thm:bounds} then implies that~$c_V(G_k)$, $c_E(G_k)$ and~$c_{E,\mathrm{r}}(G_k)$ are unbounded as well for growing~$k$.

The construction of~$G_k$ is quite intricate.
We divide it into several steps.

\subparagraph{The Graph \texorpdfstring{$\bm{H[s]}$}{H[s]}.}
Let~$s \geq 1$ be an integer and let~$k = 2^{s-1}$.
We start with a graph~$H[s]$, which we call the \emph{base graph}, that has the following properties:
\begin{itemize}
    \item $H[s]$ is $2k$-regular, i.e., every vertex of~$H[s]$ has degree $2k$.
    \item $H[s]$ has girth at least~$5$.
\end{itemize}
There are many ways to construct such graphs $H[s]$, one being our graphs in \cref{sec:latin_squares} constructed from $2^s-1$ mutually orthogonal Latin squares.
An alternative construction (not relying on non-trivial tools) is an iteration of the $6$-cycle, as illustrated in \cref{fig:C6-iteration}.
We additionally endow~$H[s]$ with an orientation such that each vertex has exactly~$k = 2^{s-1}$ outgoing and exactly~$k = 2^{s-1}$ incoming edges.
(For example, orient the edges according to an Eulerian tour in~$H[s]$.)

\begin{figure}[bt]
    \centering
    \includegraphics{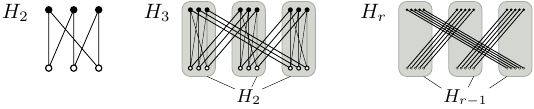}
    \caption{Iterating $C_6 = H_2$ to obtain $r$-regular (bipartite) graphs $H_r$ with $\girth(H_r) \geq 5$.}
    \label{fig:C6-iteration}
\end{figure}

\subparagraph{The Graph \texorpdfstring{\bm{$H[s,\ell]$}}{H[s,l]}.}
Let~$\ell \geq 1$ be another integer\footnote{
    We shall choose~$\ell \gg s$ later.
    So you may think of~$s$ as \enquote{short} and of~$\ell$ as \enquote{long}.
}.
We define a graph~$H[s,\ell]$ based on~$H[s]$ and its orientation.
See \cref{fig:Hsl-construction} for an illustration with $s=3$ and $\ell=4$.
For each vertex~$a$ in~$H[s]$, take a complete balanced binary tree~$T(a)$ of height~$s = \log_2(k)+1$ with root~$r(a)$ and~$2^s = 2k$ leaves.
Let~$r^{\mathrm{in}}(a)$ and~$r^{\mathrm{out}}(a)$ denote the two children of~$r(a)$ in~$T(a)$, and let~$T^{\mathrm{in}}(a)$ and~$T^{\mathrm{out}}(a)$ denote the subtrees rooted at~$r^{\mathrm{in}}(a)$ and~$r^{\mathrm{out}}(a)$, respectively.
We associate each of the~$k = 2^{s-1}$ leaves in~$T^{\mathrm{in}}(a)$ with one of the~$k$ incoming edges at~$a$ in~$H[s]$, and each of the~$k$ leaves in~$T^{\mathrm{out}}(a)$ with one of the~$k$ outgoing edges at~$a$ in~$H[s]$.
Finally, for each edge~$ab$ in~$H[s]$ oriented from~$a$ to~$b$, connect the associated leaf in~$T^{\mathrm{out}}(a)$ with the associated leaf in~$T^{\mathrm{in}}(b)$ by a path~$P(ab)$ of length~$2\ell+1$, i.e., on~$2\ell$ new inner vertices.
This completes the construction of~$H[s,\ell]$.

\medskip
\noindent
Establishing some notation and properties of~$H[s,\ell]$, observe that the following holds for any edge~$ab$ of the base graph~$H[s]$:
\begin{equation}
    \label{eqn:adjacent_roots_close}
    \dist_{H[s,\ell]}(r(a),r(b)) = 2s + 2\ell + 1
\end{equation}
Let $R = \{r(a) \mid a \in V(H[s])\}$, and call the elements of~$R$ the \emph{roots} in~$H[s,\ell]$.
For any root~$r \in R$ let us define the \emph{ball around~$r$} by
\[
    B(r) = \{
        v \in V(H[s,\ell]) \mid
        \dist_{H[s,\ell]}(v,r) \leq s + \ell
    \}
    \text{,}
\]
i.e.,~$B(r)$ consists of all vertices in~$H[s,\ell]$ that are closer to~$r$ than to any other root in~$R$.
Note that every vertex of~$H[s,\ell]$ lies in the ball around exactly one root.
Further, for distinct~$a,b \in V(H[s])$ with~$ab \notin E(H[s])$, let~$v \in B(r(a))$ and~$w \in B(r(b))$.
Then, we have
\begin{align}
    \label{eqn:distance_btw_Voronoi}
    \dist_{H[s,\ell]}(v,w) > 2\ell
    \text{.}
\end{align}
For each edge~$ab$ in~$H[s]$, let~$e(ab)$ denote the unique middle edge of~$P(ab)$ in~$H[s,\ell]$.
Then, $e(ab)$ connects a vertex in~$B(r(a))$, which we denote by~$v(a,b)$, with a vertex in~$B(r(b))$, which we denote by~$v(b,a)$.
See also \cref{fig:Hsl-construction}.

\begin{figure}[tb]
    \centering
    \includegraphics{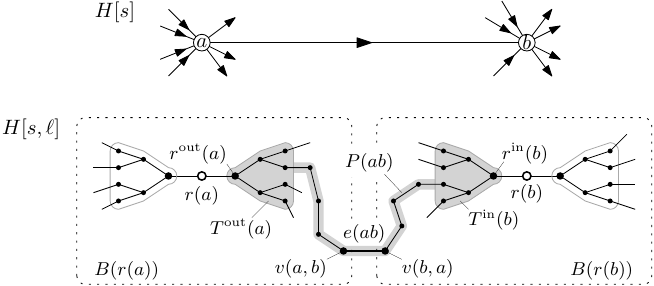}
    \caption{
        Construction of~$H[s,\ell]$ based on~$H[s]$.
        A directed edge~$ab$ in~$H[s]$ and the corresponding trees~$T^\mathrm{out}(a)$, $T^\mathrm{in}(b)$, and path~$P(ab)$ with middle edge~$e(ab)$ in~$H[s,\ell]$.
    }
    \label{fig:Hsl-construction}
\end{figure}

\subparagraph{The Graph \texorpdfstring{\bm{$H[s,\ell,m]$}}{H[s,l,m]}.}
Let~$m \geq 1$ be yet another integer\footnote{
    We shall choose~$\ell \gg m \gg s$ later.
    So you may think of~$m$ as \enquote{medium}.
}.
We start with two vertex-disjoint copies~$H_1$ and~$H_2$ of~$H[s,\ell]$, and transfer our notation (such as~$R$, $r(a)$, $v(a,b)$, etc.) for~$H[s,\ell]$ to~$H_i$ for~$i \in \{1,2\}$ by adding the subscript~$i$, e.g., $R_1$, $r_2(a)$, or $v_1(a,b)$.
We connect~$H_1$ and~$H_2$ as follows:
For each edge~$ab$ in~$H[s]$, we identify the edge~$e_1(ab)$ in~$H_1$ with the edge~$e_2(ab)$ in~$H_2$ such that $v_1(a,b) = v_2(a,b)$ and $v_1(b,a) = v_2(b,a)$.
Next, for each vertex~$a$ of~$H[s]$, use the vertex~$r_1(a)$ in~$H_1$ as an endpoint for a new path~$Q(a)$ of length~$m$, and denote the other endpoint of~$Q(a)$ by~$q(a)$.
Note that we do this only for the roots in~$H_1$.

Finally, we connect the vertices $\{q(a) \mid a \in V(H[s])\}$ by a cycle~$C$ of length $\abs{V(H[s])}$.
This completes the construction of~$H[s,\ell,m]$.
See \cref{fig:Hslm-construction} for an illustration.
Note that $H[s,\ell,m]$ has maximum degree~$3$ and degeneracy~$2$.

\begin{figure}[tb]
    \centering
    \includegraphics{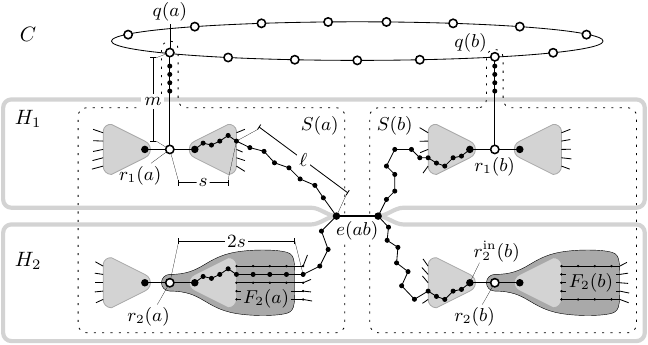}
    \caption{
        Construction of~$H[s,\ell,m]$ based on two copies of~$H[s,\ell]$.
        A directed edge~$ab$ in~$H[s]$ and the corresponding sets~$S(a)$, $S(b)$, $F_2(a)$, etc. and edge~$e(ab)$ in~$H[s,\ell,m]$.
    }
    \label{fig:Hslm-construction}
\end{figure}

\medskip
\noindent
Establishing more notation for $H[s,\ell,m]$, for each vertex~$a$ in~$H[s]$, let us define
\[
    S(a) = B(r_1(a)) \cup B(r_2(a)) \cup Q(a)
    \text{,}
\]
see again \cref{fig:Hslm-construction} for an illustration.
Let~$H_{-}$ denote the graph $H[s,\ell,m] - E(C)$ obtained from~$H[s,\ell,m]$ by removing all edges of the cycle~$C$.

\begin{lemma}
    \label{lem:Hslm_cop_number_2}
    For every $s,m \geq 1$ and $\ell > \abs{V(H[s])} + m + s$, it holds that $c(H[s,\ell,m]) \leq 2$.
\end{lemma}

Before starting with the formal proof, let us already give a vague idea of the winning strategy for two cops.
Two cops could easily capture the robber on the cycle~$C$ (including the attached paths~$Q(a)$ for~$a \in V(H[s])$).
Thus, they force him to \enquote{flee} to~$H_1$ at some point.
In a second phase, they can even force him to flee to~$H_2$.
Loosely speaking, cop~$c_1$ stays on~$C$ to prevent the robber from getting back on $C$, while cop~$c_2$ always moves towards the robber in~$H[s,\ell,m]-E(C)$.
Whenever the robber traverses one of the long paths~$P_1(ab)$ for some~$ab \in E(H[s])$, say from~$T_1(a)$ towards~$T_1(b)$, then~$c_1$ can go in $\abs{V(C)} + m + s < \ell$ steps along~$C$ to~$q(b)$ and along~$Q(b)$ and through~$T_1(b)$ to arrive at the other end of~$P_1(ab)$ before the robber.
However, to escape~$c_2$, the robber must traverse a path~$P_1(ab)$ eventually, with his only way to escape being to turn into~$H_2$ at the middle edge~$e(ab)$.
At some point, we enforce the situation that the robber occupies some vertex~$v$ of~$H_2$ and one of the cops, say~$c_1$, occupies the corresponding copy of~$v$ in~$H_1$.
Now, in a third phase, the robber moves in~$H_2$ while~$c_1$ always copies his moves in~$H_1$.
This prohibits the robber to ever walk along a middle edge~$e(ab)$ for some~$ab \in E(H[s])$.
But without these edges,~$H_2$ is a forest, and thus cop~$c_2$ can capture the robber in the tree component in which the robber is located.

\begin{proof}[Proof of \cref{lem:Hslm_cop_number_2}]
    We describe a strategy for two cops, the first cop~$c_1$ and the second cop~$c_2$, to capture the robber in the classical game played on~$H[s,\ell,m]$.
    They both start (anywhere) on the cycle~$C$.
    There are three phases, each eventually ensuring a particular configuration of the game after the cops' turn.

    We start with a convenient definition and observation:
    A cop \emph{guards the robber on~$C$} if, for some vertex~$a$ in~$H[s]$,  the robber occupies some vertex $w \in Q(a) \cup T_1(a) \cup T_2(a)$ and the cop occupies the corresponding vertex~$q(a)$ on~$C$.
    Observe that if a cop already guards the robber on~$C$, then the cop can maintain this property by always moving along~$C$ towards the vertex~$q(b)$ for which the robber occupies some vertex in~$S(b)$.
    In fact, by Equation~\eqref{eqn:distance_btw_Voronoi}, we have~$\dist_{H_{-}}(v,w) > 2\ell$ if~$v \in S(a)$ and~$w \in S(b)$ with~$a \neq b$ and~$ab \notin E(H[s])$.
    Also, if~$ab$ is an edge in~$H[s]$, then the distance in~$H_{-}$ between $T_1(a) \cup T_2(a)$ and $T_1(b) \cup T_2(b)$ is equal to the length of~$P_i(ab)$, which is $2\ell+1 > 2\ell$.
    On the other hand, $\dist_C(q(a),q(b)) \leq \abs{V(C)} = \abs{V(H[s])} \leq \ell$, and hence the cop is fast enough on~$C$ in both cases.

    \proofsubparagraph{First Phase.}
    The target configuration of the first phase is that both cops guard the robber on~$C$.
    The cops can achieve this as follows:
    After both starting at the same vertex~$q$ on~$C$, cop~$c_1$ moves along~$C$ to decrease its distance to the robber (while staying on~$C$).
    The second cop~$c_2$ stays at~$q$, which forces the robber to eventually leave the cycle~$C$.
    After finitely many steps,~$c_1$ guards the robber on~$C$ and continues to maintain this property.
    This prevents the robber from ever coming back to~$C$ without being captured.
    Then, cop~$c_2$ can move along~$C$ towards the first cop to eventually reach the target configuration.

    \proofsubparagraph{Second Phase.}
    The target configuration of the second phase is that the robber occupies some vertex~$w_2$ in~$H_2$, and one cop occupies the corresponding copy~$w_1$ of~$w_2$ in~$H_1$.
    To this end, $c_1$ starts by maintaining the guarding property, whereas~$c_2$ leaves~$C$ and always goes towards the robber along a shortest path in~$H_{-}$.
    Whenever the robber does not move or goes back and forth along an edge (i.e., is at the same vertex as two moves ago), then~$c_2$ reduces its distance to the robber.
    As this may happen only a bounded number of times without the robber being captured by~$c_2$, we may assume that the robber always \enquote{moves forward} and thus eventually occupies one of the two vertices of the middle edge $e_1(ab) = e_2(ab)$ corresponding to some edge~$ab$ in~$H[s]$.
    Moreover, the robber continues along~$P_1(ab)$ or~$P_2(ab)$ towards one of~$T_1(a)$, $T_2(a)$, $T_1(b)$ or~$T_2(b)$; say towards~$T_i(b)$ for some~$i \in \{1,2\}$.

    At this point, cop~$c_1$ goes along~$C$ to~$q(b)$, then along~$Q(b)$ to~$r_1(b)$, and along~$T_1(b)$ to the leaf~$w_1$ of~$T_1(b)$ that is the endpoint of~$P_1(ab)$.
    This takes~$c_1$ at most $\abs{V(C)} + m + s < \ell$ steps.
    On the other hand, the robber needs at least~$\ell$ steps to reach~$T_1(b) \cup T_2(b)$.
    This way, the robber (always moving forward) is either captured by~$c_1$ on~$w_1$, reaches the copy~$w_2$ of~$w_1$ in~$H_2$, or the copy in~$H_2$ of the vertex occupied by~$c_2$ in~$H_1$, which is the desired target configuration.

    \proofsubparagraph{Third Phase.}
    Finally, we shall capture the robber.
    Without loss of generality, let~$c_1$ be the cop in~$H_1$ that copies the robber's moves in~$H_2$.
    Cop~$c_2$ continues to go towards the robber along a shortest path in~$H_{-}$, forcing him to keep moving forward.
    As $H_2 - \{e_2(ab) \mid ab \in E(H[s])\}$ is a forest, the robber must eventually occupy an endpoint of~$e_2(ab)$ for some $ab \in E(H[s])$, where he is captured by~$c_1$.
\end{proof}

In contrast to the bounded classical cup number, the surrounding cop numbers are unbounded:

\begin{lemma}
    \label{lem:Hslm_surrounding_cop_number_k}
    For every $s \geq 1$, $m > 2s+1$, and $\ell > 3s+1$, it holds that $c_{V, \mathrm{r}}(H[s,\ell,m]) \geq k = 2^{s-1}$.
\end{lemma}

Again, we start by giving a vague idea how the robber can avoid being surrounded by~$k-1$ vertex cops:
The robber stays solely in~$H_2$, and he moves between roots~$r_2(x)$ for~$x \in V(H[s])$, always from a root~$r_2(a)$ to the next root~$r_2(b)$ for which the edge~$ab$ in~$H[s]$ is directed from~$a$ to~$b$.
In~$H[s]$, vertex~$a$ has~$k$ outgoing neighbors, and we show that at least one such neighbor~$b$ is always \enquote{safe} for the robber to escape to.
However, it is quite tricky to identify this safe neighbor.
Indeed, the robber has to start moving in the \enquote{right} direction down the tree~$T_2(a)$ always observing the cops' response, before he can be absolutely sure which outgoing neighbor~$b$ of~$a$ is safe.
With suitable choices of~$s$, $\ell$ and~$m$, the robber is fast enough at~$r_2(b)$ to then choose his next destination from there.
The crucial point is that the cops can \enquote{join} the robber when he traverses the middle edge~$e(ab)$, but they can never be one step ahead of the robber on~$P_2(ab)$; and thus never surround him.

\begin{proof}[Proof of \cref{lem:Hslm_surrounding_cop_number_k}]
    We provide a strategy for the robber to avoid getting surrounded indefinitely against any strategy for~$k-1$ cops on~$H[s,\ell,m]$.
    Our strategy ensures that the robber always moves onto a vertex that is not occupied by a cop, and thus works for the restrictive vertex version of the game.
    For this strategy, the robber shall exclusively stay on~$H_2$ in~$H[s,\ell,m]$.
    Recall that for any vertex~$a$ of~$H[s]$, there is a corresponding binary tree~$T_2(a)$ with root~$r_2(a)$, and two subtrees~$T_2^{\mathrm{in}}(a)$ and~$T_2^{\mathrm{out}}(a)$ rooted at~$r_2^{\mathrm{in}}(a)$ and~$r_2^{\mathrm{out}}(a)$, respectively.
    The leaves of~$T_2^{\mathrm{out}}(a)$ and~$T_2^{\mathrm{in}}(a)$ correspond to the outgoing, respectively incoming, edges at~$a$ in~$H[s]$, and for every directed edge from~$a$ to~$b$ in~$H[s]$ there is a path~$P_2(ab)$ of length $2\ell+1$ from a leaf in~$T_2^{\mathrm{out}}(a)$ to a leaf in~$T_2^{\mathrm{in}}(b)$.

    For our strategy, we call a configuration of the game after the robber's turn a \emph{good situation} if there is a vertex~$a$ in~$H[s]$ such that
    \begin{itemize}
        \item the robber occupies~$r_2(a)$~in $H[s,\ell,m]$ and
        \item each cop has distance at least~$2s$ to~$r_2(a)$ in graph $H[s,\ell,m] - \bigl\{\{r_2(a), r_2^{\mathrm{in}}(a)\}\bigr\}$, obtained by removing only the edge between~$r_2(a)$ and~$r_2^{\mathrm{in}}(a)$.
    \end{itemize}
    
    For convenience, let us denote by~$F_i(a)$ (with $i \in \{1,2\}$) the set of all vertices in~$H[s,\ell,m]$ at distance less than~$2s$ to~$r_i(a)$ in $H[s,\ell,m] - \bigl\{\{r_i(a), r_i^{\mathrm{in}}(a)\}\bigr\}$.
    Thus, we have a good situation with the robber at~$r_2(a)$ if no cop occupies a vertex of~$F_2(a)$.
    Since $\ell > s$, we have $F_1(a),F_2(a) \subset S(a)$ for any~$a \in V(H[s])$.
    Together with $\abs{V(H[s])} > 2k$, this implies that the robber can start the game in a good situation, i.e., on a vertex~$r_2(a)$ with no cop in~$F_2(a)$ (as there are only~$k-1$ cops).
 
    Observe that the robber is not surrounded in a good situation because there is no cop on the neighbor~$r_2^{\mathrm{out}}(a)$ of~$r_2(a)$.
    Recall that~$H[s]$ is endowed with an orientation with outdegree and indegree exactly~$k$ at each vertex~$a$.
    Let~$N_{H[s]}^+(a)$ denote the set of all outgoing neighbors of~$a$ in~$H[s]$.
    It is then enough to show that, starting in a good situation at vertex~$r_2(a)$, the robber can directly move from~$r_2(a)$ to~$r_2(b)$ for some $b \in N_{H[s]}^+(a)$ in $2s + 2\ell + 1$ steps (along~$T_2^{\mathrm{out}}(a)$, then~$P_2(ab)$, then~$T_2^{\mathrm{in}}(b)$, cf.~\eqref{eqn:adjacent_roots_close}) to be again in a good situation at vertex~$r_2(b)$ and without being surrounded at any time.
    In fact, it will be enough to ensure that no cop reaches~$F_2(b)$ before the robber reaches~$r_2(b)$.
    (Recall that from a good situation, the cops go first.)

    \begin{claim}
        Let $a$ be a vertex in $H[s]$ and $b,b' \in N^+_{H[s]}(a)$ be two distinct outgoing neighbors of $a$ in $H[s]$.
        Then, there is no vertex $w$ in $H[s,\ell,m] - S(a)$ that has distance at most $2\ell + 2s + 1$ to $F_2(b)$ and distance at most $2\ell + 2s + 1$ to $F_2(b')$.
    \end{claim}
    
    \begin{claimproof}
        Let $w$ be any fixed vertex in $H[s,\ell,m] - S(a)$.
        We have to show that $w$ has distance greater than $2\ell + 2s + 1$ to $F_2(b)$ or $F_2(b')$, or both.
        
        First, observe that a shortest path from~$F_2(b)$ to the cycle~$C$ uses~$\ell-s$ edges of~$P_2(bc)$ for some~$c \in N_{H[s]}(b)$, $\ell$ edges of~$P_1(bc)$, $s$ edges of~$T_1(b)$, and~$m$ edges of~$Q(b)$, for a total of $(\ell-s)+\ell+s+m = 2\ell + m > 2\ell + 2s + 1$ edges.
        Thus, we are done in case $w \in V(C)$.
        Furthermore, if a shortest path from $w$ to $F_2(b)$ (respectively $F_2(b')$) contains a vertex in $C$, then $w$ has distance greater than $2\ell + 2s + 1$ to $F_2(b)$ (respectively $F_2(b')$), and we are done for such $w$ as well.

        As $H[s]$ has girth at least $5$, we know that $b$ and $b'$ are not adjacent in $H[s]$.
        Furthermore, we have $N_{H[s]}(b) \cap N_{H[s]}(b') = \{a\}$, i.e., $a$ is the only common neighbor of $b$ and $b'$ in $H[s]$.
        As $w \in S(c)$ for some vertex $c \neq a$ in $H[s]$, we can assume without loss of generality that $c \neq b$ and $c$ is not adjacent to $b$ in $H[s]$.
        Now consider a shortest path $P$ from $w$ to $F_2(b)$ in $H[s,\ell,m]$.
        We already know that $P$ contains no vertex in $C$.
        As $c$ and $b$ are not adjacent in $H[s]$, the path $P$ starts in $S(c)$, ends in $S(b)$, and traverses $S(d)$ for some vertex $d\in N_{H[s]}(b)$.
        But then $P$ contains at least $2\ell$ vertices that are in $S(d)$ in order to traverse $S(d)$.
        Together with the at least $\ell - s$ vertices in $S(b)$ (to reach $F_2(b)$) we see that $w$ has distance at least $3\ell - s > 2\ell + 2s + 1$ to $F_2(b)$.
        This proves the claim.
    \end{claimproof}

    Recall that we are in a good situation with the robber at~$r_2(a)$, and thus no cop occupies a vertex in~$F_2(a)$. 
    We seek to move the robber directly towards~$r_2(b)$ for some~$b \in N^+_{H[s]}(a)$. 
    Guided by the above claim, we say that a cop occupying vertex~$w$ in~$H[s,\ell,m]$ \emph{blocks} vertex~$b \in N_{H[s]}^+(a)$ if~$w$ has distance at most~$2\ell + 2s + 1$ to~$F_2(b)$ and~$w \notin F_1(a)$.
    Note that with this definition, the above claim implies that each cop blocks at most one~$b \in N^+_{H[s]}(a)$.
    Let~$X \subset N^+_{H[s]}(a)$ be the (possibly empty) subset of all outgoing neighbors of~$a$ that are blocked by some cop.

    There are~$k-1$ cops in total, hence~$\abs{X} \leq k-1$, while at the same time~$\abs{N^+_{H[s]}(a)} = k$.
    Thus, we have $N^+_{H[s]}(a) - X \neq \emptyset$.
    However, we shall make a definite decision for a vertex $b \in N^+_{H[s]}(a) - X$ only after the first~$s$ steps. This is due to the cops in~$F_1(a)$ with distance at most $2\ell + 2s + 1$ to some vertex in $N^+_{H[s]}(a) - X$.
    We choose each move for the robber inside~$T_2(a)$ according to the locations of the cops in~$F_1(a)$ at that moment.
    As the first step (after the cops have moved), the robber moves to~$r_2^{\mathrm{out}}(a)$, the root of~$T_2^{\mathrm{out}}(a)$.

    Now consider the situation after step~$j$ of the robber, $j \in \{1, \ldots, s\}$.
    Let the robber occupy the non-leaf vertex~$v_2$ in~$T_2^{\mathrm{out}}(a)$ at distance~$j-1$ from the root~$r_2^{\mathrm{out}}(a)$.
    Consider the copy~$v_1$ of~$v_2$ in~$T_1^{\mathrm{out}}(a)$, and the subtree~$T$ of~$T_1^{\mathrm{out}}(a)$ below~$v_1$ with exactly~$2^{s-j}$ leaves.
    We call a vertex~$b \in N^+_{H[s]}(a)$ \emph{available} if~$b$ is not blocked ($b \notin X$) and~$P_1(ab)$ starts at a leaf of~$T$, and denote by~$Y$ the set of all available vertices.
    We maintain the invariant that after step~$j$ of the robber, strictly less than~$\abs{Y}$ cops occupy a vertex in
    \[
        U = \{v_2\} \cup V(T) \cup \bigcup_{b \in Y} P_1(ab)
        \text{,}
    \]
    meaning that there are more available vertices than cops below~$v_1$.
    Observe that a cop blocking some~$b \in X$ starts at distance at least~$s$ to~$T$ and cannot occupy a vertex in~$U$ after~$j$ steps.
    In particular, this implies that the invariant holds for~$j = 1$, as there are less than~$k$ cops in total, and thus less than $k - \abs{X} = 2^{s-j} - \abs{X} = \abs{Y}$ cops in~$U$.  

    After the next move of the cops, we shall choose the next step of the robber.
    Consider the two subtrees~$T'$ and~$T''$ below~$v_1$, let~$x'$ and~$x''$ be their respective number of leaves corresponding to blocked vertices, and let~$y'$ and~$y''$ be their respective number of leaves corresponding to available vertices.
    By our invariant, there are now (after the cops' move) less than $y' + y'' = 2^{s-j} - (x' + x'')$ cops in~$U - \{v_2\}$.
    Thus, there are less than $y' = 2^{s-j-1} - x'$ cops in the direction of~$T'$, or less than $y'' = 2^{s-j-1} - x''$ cops in the direction of~$T''$ (or both).
    In particular, there is a move for the robber in~$T_2^{\mathrm{out}}(a)$ to maintain our invariant.
    
    After~$s$ moves, the robber occupies some leaf~$v_2$ of~$T_2^{\mathrm{out}}(a)$.
    By our invariant,~$v_2$ is the first vertex of the path~$P_2(ab)$ for some vertex~$b \notin X$ and there is no cop (\enquote{less than~$1$}) in the set~$U = P_1(ab)$.
    Thus, the robber can move along~$P_2(ab)$ in~$2\ell + 1 + s$ further steps to reach~$r_2(b)$ before any of the cops can.
    Finally, if the cops could surround the robber on the way, then at least one cop could also arrive at~$r_2(b)$ before the robber, which we just excluded.
    Hence, we can guarantee a good situation at~$r_2(b)$, which concludes the proof.
\end{proof}

Finally, \cref{lem:Hslm_cop_number_2,lem:Hslm_surrounding_cop_number_k,thm:bounds} immediately give the following corollary, which proves \cref{thm:surrounding_vs_classic}.

\begin{corollary}
    For any $s \geq 1$, $k = 2^{s-1}$, $m \geq 2s+1$, and $\ell \geq \abs{V(H[s])} + m + s$, the graph $G_k = H[s,\ell,m]$ has~$\Delta(G_k) = 3$ and
    \[
        c(G_k) \leq 2 \text{,}
        \quad c_V(G_k) \geq k \text{,}
        \quad c_{V, \mathrm{r}}(G_k) \geq k \text{,}
        \quad c_{E, \mathrm{r}}(G_k) \geq \frac{k}{2} \quad \text{and}
        \quad c_E(G_k) \geq \frac{k}{2}         
        \text{.}
    \]
\end{corollary}

\section{Conclusion}
We considered the cop numbers of four different surrounding versions of the well-known Cops and Robber game on a graph~$G$, namely $c_V(G)$, $c_{V, \mathrm{r}}(G)$, $c_E(G)$ and $c_{E, \mathrm{r}}(G)$.
Here, index \enquote{$V$} denotes the vertex versions, while index \enquote{$E$} denotes the edge versions, i.e., whether the cops occupy the vertices or the edges of the graph (recall that the robber always occupies a vertex).
An additional index \enquote{$\mathrm{r}$} stands for the corresponding restrictive version, meaning that the robber must not end his turn on a cop or move along an edge occupied by a cop.

Only the two restrictive cop numbers have recently been considered in the literature, the vertex version~$c_{V, \mathrm{r}}(G)$ in~\cite{Bradshaw2019_SurroundingBoundedGenus, Burgess2020_Surround} (denoted by~$\sigma(G)$ and~$s(G)$) and the edge version~$c_{E, \mathrm{r}}(G)$ in~\cite{Crytser2020_Containment, Pralat2015_Containment} (denoted by~$\xi(G)$).

In this paper, we related the four different versions to each other, showing that all of them lie within a factor of (at most)~$2\Delta(G)$ to each other.
We proved that this is tight (up to small additive or multiplicative constants) for all combinations by presenting explicit graph families.
It is an interesting open question to identify the exact constant factors for the lower and upper bounds in \cref{thm:bounds}.
We conjecture that all six presented upper bounds are tight (up to small \emph{additive} constants).
This would mean that optimal strategies for the cops in one surrounding version can indeed be obtained by simulating optimal strategies from different surrounding versions.

As a second main result, we disproved a conjecture by Crytser, Komarov and Mackey~\cite{Crytser2020_Containment} by constructing a family of graphs with maximum degree~$3$ in which the classical cop number is bounded whereas the cop number in all four surrounding versions is unbounded.
It remains open to find other parameters that can be used to bound the surrounding cop numbers from above in combination with the classical cop number.

\bibliographystyle{plainurl}
\bibliography{references}

\end{document}